\documentclass[12pt]{article}
\usepackage[a4paper,margin=1.3in]{geometry}
\usepackage{amssymb}
\usepackage{amsmath}
\usepackage{amsfonts}
\usepackage{amsthm}

\usepackage[all]{xy}

\newcommand{\C}{\mathbb{C}}
\newcommand{\Z}{\mathbb{Z}}

\newcommand{\G}{\mathcal{G}}
\newcommand{\E}{\mathcal{E}}

\newcommand{\Odd}{\mathcal{O}}
\newcommand{\Mod}[1]{\ (\text{mod}\ #1)}

\newtheorem{theorem}{Theorem}

\newtheorem{conjecture}{Conjecture}
\newtheorem{corollary}{Corollary}

\newtheorem{definition}{Definition}
\newtheorem{example}{Example}

\newtheorem{proposition}{Proposition}

\numberwithin{equation}{section}

\begin{document}

\title{\textbf{Schur ring over group $\Z_{2}^{n}$, circulant $S$-sets invariant by decimation and Hadamard matrices}}
\author{Ronald Orozco L\'opez}

\newcommand{\Addresses}{{
  \bigskip
  \footnotesize

  \textsc{Department of Mathematics, Universidad de los Andes,
    Bogot\'a Colombia,}\par\nopagebreak
  \textit{E-mail address}, \texttt{rj.orozco@uniandes.edu.co}

}}

\maketitle

\begin{abstract}
In this paper a variety of issues are discussed, Schur ring, $S$-sets, circulant orbits, decimation
operator and Hadamard matrices and their relation between them is shown. Firstly we define the
complete $S$-sets. Next, we study the structure of Schur ring with circulant basic sets 
over $\Z_{2}^{n}$ and we define the free and non-free circulant $S$-sets, the symmetric,
non-symmetric and antisymmetric circulant $S$-sets. We prove that all this 
$S$-sets are invariants under decimation. Finally, we prove that if a Hadamard matrix exist then 
this is contained in a complete $S$-set. Also, we prove that can't exist circulant and
with one core Hadamard matrices with some particular structure. These theorems include a result
known on symmetric circulant Hadamard matrices of order $4n$ only when $n$ is an odd number.
\end{abstract}

\section{Introduction}

Let $G$ be a finite group with identity element $e$ and $\C[G]$ the group algebra of all formal sums 
$\sum_{g\in G}a_{g}g$, $a_{g}\in \C$, $g\in G$. For $T\subset G$, the element $\sum_{g\in T}g$ will 
be denoted by $\overline{T}$. Such an element is also called a $\textit{simple quantity}$. 
The transpose of $\overline{T} = \sum_{g\in G}a_{g}g$ is defined as $\overline{T}^{\top} = \sum_{g\in G}a_{g}(g^{-1})$. Let $\{T_{0},T_{1},...,T_{r}\}$ be a partition of $G$ and let $S$ be the subspace
of $\C[G]$ spanned by $\overline{T_{1}},\overline{T_{2}},...,\overline{T_{r}}$.  We say that $S$ is 
a $\textit{Schur ring}$ ($S$-ring, for short) over $G$ if: 

\begin{enumerate}
\item $T_{0} = \lbrace e\rbrace$, 
\item for each $i$, there is a $j$ such that $\overline{T_{i}}^{\top} = \overline{T_{j}}$,
\item for each $i$ and $j$, we have $\overline{T_{i}}\ \overline{T_{j}} = \sum_{k=1}^{r}\lambda_{i,j,k}\overline{T_{k}}$, for constants $\lambda_{i,j,k}\in\C$.
\end{enumerate}

The numbers $\lambda_{i,j,k}$ are the structure constants of $S$ with respect to the linear base 
$\{\overline{T_{0}},\overline{T_{1}},...,\overline{T_{r}}\}$. The sets $T_{i}$ are called the
\textit{basic sets} of the $S$-ring $S$. Any union of them is called an $S$-sets. Thus, 
$X\subseteq G$ is an $S$-set if and only if $\overline{X}\in S$. The set of all $S$-set is closed
with respect to taking inverse and product. Any subgroup of $G$ that is an $S$-set, is called an 
$S$-\textit{subgroup} of $G$ or $S$-\textit{group} (For details, see [1],[2],[3]). A partition
$\{T_{0},...,T_{r}\}$ of $G$ is called \textit{Schur partition} or $S$-\textit{partition} 
if the $T_{i}$ fulfill $T_{0}=\{e\}$ and $T_{i}^{-1}=\{g^{-1}:g\in T_{i}\}=T_{j}$ for each $i$ and 
for each $j$. It is known that there is a 1-1 correspondence between $S$-ring over $G$ and 
$S$-partition of $G$. By using this correspondence, in this paper we will refer to an $S$-ring by
mean of its $S$-partition.

On the other hand, a Hadamard matrix $H$ is an $n$ by $n$ matrix all of whose entries are $+1$ or 
$-1$ which satisfies $HH^{t}=nI_{n}$, where $H^{t}$ is the transpose of $H$
and $I_{n}$ is the unit matrix of order $n$. It is also known that, if an
Hadamard matrix of order $n>1$ exists, $n$ must have the value $2$ or be divisible by 4. 

There are several conjectures associated with Hadamard matrices. 
The main conjecture concerns its existence. This states that a Hadamard matrix exists for
all multiple order of 4. Another very important conjecture states that no circulant Hadamard matrix
exists if the order is different from 4. Another conjecture is known as the Structured Hadamard
Conjecture [5]. In an effort to prove all these problems people have constructed different
Hadamard matrices types: Paley type[4], with one and two circulant core [9],[10], Williamson type
[7],Goethals-Seidel type and other [6]. 

All these constructions are defined with circulant matrices. 
Therefore, a general strategy is to define an $S$-partition over $G=\Z_{2}^{n}$, where the basic sets
are circulants induced by the cyclic group $C_{n}=\left\langle C\right\rangle\leq Aut(\Z_{2}^{n})$ 
of order $n$ and $Aut(\Z_{2}^{n})$ is the group of automorphisms of $\Z_{2}^{n}$. We denote to the
$S$-partition of $\Z_{2}^{n}$ induced by a group $H\leq Aut(\Z_{2}^{n})$ by 
$\mathfrak{S}(\Z_{2}^{n},H)$. Then $\mathfrak{S}(\Z_{2}^{n},C_{n})$ will denote the $S$-partition
induced by $C_{n}$.

In this paper several Schur rings induced by permutation automorphic subgroups of 
$Aut(\Z_{2}^{n})$ are studied. Also several $S$-sets are considered, namely complete $S$-sets, 
free and non-free circulant $S$-sets, circulant $S$-sets invariant
by decimation, symmetric, non-symmetric and antisymmetric circulant $S$-sets. 
We prove that all this $S$-sets are invariants under decimation. Finally, we prove that if a 
Hadamard matrix exist, then this is contained in a complete $S$-set. Also, we prove that if 
a circulant and with one core Hadamard matrix of order $nr$ exists, then these can't be partitioned
in a basic set.

\section{Schur ring $\mathfrak{S}(\Z_{2}^{n},S_{n})$}

In this paper denote by $\Z_{2}$ the cyclic group of order 2 with elements $+$
and $-$(where + and $-$ mean 1 and $-1$ respectively). Let 
$\Z_{2}^{n}=\overset{n}{\overbrace{\Z_{2}\times \cdots \times \Z_{2}}}$. Then all $X\in\Z_{2}^{n}$ are sequences of $+$ and $-$ and will be called $\Z_{2}-$\textit{sequences}. 

Let $\omega(X)$ denote the Hamming weight of $X\in\Z_{2}^{n}$. Thus, $\omega(X)$ is the number of $+$ in any $\Z_{2}-$sequences $X$ of $\Z_{2}^{n}$. Now let $\G_{n}(k)$ be the subset of $\Z_{2}^{n}$ such that $\omega(X)=k$ for all $X\in\G_{n}(k)$, where $0\leq k\leq n$.

\begin{proposition}
Let $\G_{n}(k)\subset\Z_{2}^{n}$. Then
\begin{equation}\label{eq4}
\G_{n}(k)=\{+\} \times \G_{n-1}(k-1)\cup \{-\} \times \G_{n-1}(k)
\end{equation}
and
\begin{equation}\label{eq5}
\left\vert \G_{n}(k)\right\vert =\binom{n}{k},
\end{equation}
where $\binom{n}{k}$ are the binomial coefficients.
\end{proposition}
\begin{proof}
By induction. When $k$ is $0$ or $n$ we have
\begin{eqnarray}\label{eq6}
\G_{n}(0) &=&\{-\} \times \G_{n-1}(0), \\
\G_{n}(n) &=&\{+\} \times \G_{n-1}(n-1).
\end{eqnarray}
Suppose $1\leq k\leq n-1$. Since
\begin{eqnarray*}
\omega(\{+\} \times \G_{n-1}(k-1)) &=&\omega(\{-\} \times 
\G_{n-1}(k))=k
\end{eqnarray*}
and
\begin{equation}\label{binomial}
\binom{n-1}{k-1}+\binom{n-1}{k}=\binom{n}{k}
\end{equation}
then $\{+\} \times \G_{n-1}(k-1)\cup \{-\} \times \G_{n-1}(k)$ 
contains all $X$ of\ $\Z_{2}^{n}$ such that $\omega(X) =k$. Hence
\begin{equation}\label{eq7}
\G_{n}(k)=\{+\} \times \G_{n-1}(k-1)\cup \{-\} \times \G_{n-1}(k).
\end{equation}
and $\left\vert \G_{n}(k)\right\vert$ follows from (\ref{binomial}).
\end{proof}

We let $T_{i}=\G_{n}(n-i)$. It is straightforward to prove that the partition 
$\mathfrak{S}(\Z_{2}^{n},S_{n})=\{\G_{n}(0),...,\G_{n}(n)\}$ induces an $S$-partition over 
$\Z_{2}^{n}$, where $S_{n}\leq Aut(\Z_{2}^{n})$ is the permutation group on $n$ objects. From [3] 
it is know that the constant structure $\lambda_{i,j,k}$ is equal to
\begin{equation}\label{estruc_cons}
\lambda_{i,j,k}=
\begin{cases}
0&\mbox{if } i+j-k\ \mbox{is an odd number}\\
\binom{k}{(j-i+k)/2}\binom{n-k}{(j+i-k)/2} &\mbox{if } i+j-k\ \mbox{is an even number}
\end{cases}
\end{equation}

From (\ref{estruc_cons}) is follows that
\begingroup\makeatletter\def\f@size{10}\check@mathfonts
\begin{equation}\label{producto}
\G_{n}(a)\G_{n}(b)=
\begin{cases}
\bigcup\limits_{i=0}^{a}\G_{n}(n-a-b+2i), & 0\leq a\leq \left[\dfrac{n}{2}\right], a\leq b\leq n-a,\\
\bigcup\limits_{i=0}^{n-a}\G_{n}(a+b-n+2i), & \left[\dfrac{n}{2}\right]+1\leq a\leq n, n-a\leq b\leq a.  
\end{cases}
\end{equation}
\endgroup
A proof by induction of (\ref{producto}) is presented in the last section of this paper.

It follows directly from (\ref{estruc_cons}) that $\lambda_{i,j,2k+1}=0$ if $i+j$ is even and 
$\lambda_{i,j,2k}=0$ if $i+j$ is odd. The union of all basic sets $\G_{n}(2a)$ in $S$ will be denoted 
by $\E_{n}$ and the union of all basic sets $\G_{n}(2a+1)$ in $S$ will be denoted $\Odd_{n}$. 
The sets $\E_{2n}$ and $\Odd_{2n+1}$ are subgroups of order $2^{2n-1}$ and $2^{2n}$, respectively.  
Then $$\mathfrak{S}(\E_{2n},S_{n})=\{\G_{2n}(0),\G_{2n}(2),...,\G_{2n}(2n)\}$$ and 
$$\mathfrak{S}(\Odd_{2n+1},S_{n})=\{\G_{2n+1}(1),\G_{2n+1}(3),...,\G_{2n+1}(2n+1)\}$$
are $S$-subgroups of $\Z_{2}^{2n}$ and $\Z_{2}^{2n+1}$, respectively.

\begin{theorem}
Let $\E_{2n},\Odd_{2n+1}$ subgroups of $\Z_{2}^{n}$. Then
\begin{enumerate}
\item $\E_{2n}=\G_{2n}(n)^{2}$.
\item $\Odd_{2n+1} = \G_{2n+1}(n)^{2}$.
\end{enumerate}
\end{theorem}
\begin{proof}
The statements are follow of (\ref{producto}).
\end{proof}

\section{Complete maximal $S$-set}

In this section a result on $S$-sets of $\mathfrak{S}(\Z_{2}^{n},S_{n})$ will be studied. 
In particular, will be defined the complete maximal $S$-sets that will be used in a future section
to show that all Hadamard matrix is contained in such $S$-sets. Let $\mathfrak{S}^{\prime}$ denote
a set of basic sets of $\mathfrak{S}(\Z_{2}^{n},S_{n})$. By the 1-1 correspondence between 
$S$-partitions and $S$-rings we can identify $\mathfrak{S}^{\prime}$ with $S$-sets in 
$\mathfrak{S}(\Z_{2}^{n},S_{n})$.

\begin{definition}
Take $\G_{n}(a)$ in $\mathfrak{S}(\Z_{2}^{n},S_{n})$. Let $\mathfrak{S}^{\prime}\subset\mathfrak{S}(\Z_{2}^{n},S_{n})$ a set of basic sets. We will call $\mathfrak{S}^{\prime}$ a 
$\G_{n}(a)$-\textbf{complete} $S$-set if it is hold
\begin{enumerate}
\item $\G_{n}(i)\G_{n}(j)\supset\G_{n}(a)$ for all $\G_{n}(i),\G_{n}(j)\in\mathfrak{S}^{\prime}$,
\item There is no $\G_{n}(b)\in\mathfrak{S}(\Z_{2}^{n},S_{n})$ such that 
$\G_{n}(b)^{2}\supset\G_{n}(a)$ and $\G_{n}(b)\G_{n}(k)\supset\G_{n}(a)$ for all 
$\G_{n}(k)\in\mathfrak{S}^{\prime}$.
\end{enumerate}
\end{definition}

In the following theorem we will show that there is no $\G_{n}(a)$-complete for all $n$ and all $a$.

\begin{theorem}\label{the_non_exis_S_comp}
\begin{enumerate}
\item There is no $\G_{2n}(2a+1)$-complete $S$-sets in $\mathfrak{S}(\Z_{2}^{2n},S_{n})$.
\item There is no $\G_{2n+1}(2a)$-complete $S$-sets in $\mathfrak{S}(\Z_{2}^{2n+1},S_{n})$.
\end{enumerate}
\end{theorem}
\begin{proof}
From $\G_{2n}(b)^{2}\supset\G_{n}(2a+1)$ it is followed that $\vert 2n-2b\vert+2i=2a+1$, but this
is not possible. Equally for the other statement.
\end{proof}

\begin{definition}
Let $\E_{n}(a)$ denote the $\G_{n}(a)$-complete $S$-set with basic sets having Hamming weight an
even number and let $\Odd_{n}(a)$ denote the $\G_{n}(a)$-complete $S$-set with basic sets having
Hamming weight an odd number. We define the order of $\E_{n}(a)$ or $\Odd_{n}(a)$ as the number 
of basic sets contained in them.
\end{definition}

The following theorems tell us how many $\G_{n}(a)$-complete $S$-sets exists in $\Z_{2}^{n}$.

\begin{theorem}
There exists exactly $n+1$ $\G_{n}(n)$-complete $S$-sets in $\Z_{2}^{n}$.
\end{theorem}
\begin{proof}
Trivially follows of $\G_{n}(a)^{2}\supset\G_{n}(n)$ for $0\leq a\leq n$.
\end{proof}

\begin{theorem}
There exist exactly one $\G_{2n}(0)$-complete $S$-set in $\Z_{2}^{2n}$.
\end{theorem}
\begin{proof}
By (\ref{producto}) part I we have for $0\leq a\leq n$ that
\begin{equation*}
\G_{2n}(a)\G_{2n}(a)=\bigcup\limits_{j=0}^{a}\G_{2n}(2n-2a+2j)\supset\G_{2n}(0).
\end{equation*}
Then $2n-2a+2j=0$ implies that $a\geq n$. So $a=n$ and $\{\G_{2n}(n)\}$ is the only
$\G_{2n}(0)$-complete in $\Z_{2}^{2n}$.
\end{proof}

From theorem \ref{the_non_exis_S_comp} there is no $\G_{n}(n-1)$-complete $S$-set. Then, after 
excluding the previous trivial cases, we have

\begin{theorem}
There exists exactly two $\G_{n}(a)$-complete $S$-sets in $\Z_{2}^{n}$ for $1\leq a\leq n-2$, 
$(n-a)\equiv0\mod2$, of order $\frac{a}{2}$ and order $\frac{a}{2}+1$ only if $a$ is an even 
number and both with order $\frac{a+1}{2}$ only if $a$ is an odd number.
\end{theorem}
\begin{proof}
We looking for $\G_{n}(a)$-complete $S$-sets in $\mathfrak{S}(\Z_{2}^{n},S_{n})$. By
(\ref{producto}) part I we have that if 

\begin{equation*}
\G_{n}(b)\G_{n}(b)=\bigcup\limits_{j=0}^{b}\G_{n}(n-2b+2j)\supset\G_{n}(a),
\end{equation*}

then $n-2b+2j=a$ only if $b\geq\frac{n-a}{2}$. Hence 
$\frac{n-a}{2}\leq b\leq\lceil\frac{n}{2}\rceil$. 
Also

\begin{equation*}
\G_{n}(b)\G_{n}(c)=\bigcup\limits_{j=0}^{b}\G_{n}(n-b-c+2j)\supset \G_{n}(a)
\end{equation*}

only if $b\leq c\leq n-b$ and $b+c=2q$, $q>0$. 
On the other hand, by (\ref{producto}) part II

\begin{equation*}
\G_{n}(b)\G_{n}(b)=\bigcup\limits_{j=0}^{n-b}\G_{n}(2b-n+2j)\supset \G_{n}(a)
\end{equation*}

only if $\lceil\frac{n}{2}\rceil+1\leq b\leq\frac{a+n}{2}$. And

\begin{equation*}
\G_{n}(b)\G_{n}(c)=\bigcup\limits_{j=0}^{n-b}\G_{n}(b+c-n+2j)\supset \G_{n}(a),
\end{equation*}

only if $n-b\leq c\leq b$. Then it follows that $b,c\in\left[\frac{n-a}{2},\frac{n+a}{2}\right]$.
Therefore if $a$ is an odd number, then there are sets $\E_{n}(a)$ and $\Odd_{n}(a)$ both of order 
$\frac{a+1}{2}$. If $a$ is an even number, then there are sets $\E_{n}(a)$ and $\Odd_{n}(a)$
of order $\frac{a}{2}$ and $\frac{a}{2}+1$. 
\end{proof}

\begin{definition}
We will say that a $\G_{2n}(a)$-complete $S$-set is a \textbf{complete even maximal} $S$-set if 
$\G_{2n}(a)$ is some of the following basic sets: $\G_{4m-2}(2m-2)$, $\G_{4m-2}(2m)$, 
$\G_{4m}(2m)$. Equally, we will say that a $\G_{2n-1}(a)$-complete $S$-set is a 
\textbf{complete odd maximal} $S$-set if $\G_{2n-1}(a)$ is some of the following basic sets: 
$\G_{4m-3}(2m-1)$, $\G_{4m-1}(2m-1)$.
\end{definition}

The following corollaries follow of the theorem and definition above

\begin{corollary}
There exists exactly two complete odd maximal $S$-sets, both with order $n$, in $\Z_{2}^{4n-1}$
\end{corollary}

\begin{corollary}
There exists exactly two complete odd maximal $S$-sets, both with order $n$, in $\Z_{2}^{4n-3}$
\end{corollary}

\begin{corollary}
There exists exactly two complete even maximal $S$-sets of order $n$ and $n+1$ in $\Z_{2}^{4n}$
\end{corollary}

\begin{corollary}
There exists exactly two $\G_{4n-2}(2n-2)$-complete $S$-sets of order $n$ and $n-1$, in 
$\Z_{2}^{4n-2}$.
\end{corollary}

\begin{corollary}
There exists exactly two $\G_{4n-2}(2n)$-complete $S$-sets of order $n$ and $n+1$, in 
$\Z_{2}^{4n-2}$.
\end{corollary}

For being the Hadamard matrices of order $4n$, we are interested only in complete even maximal
$S$-sets in $\Z_{2}^{4n}$. Next, we show the complete even maximal $S$-sets in $\Z_{2}^{4n}$ for
$n=1,2,3$

\begin{example}
The following are complete maximal $S$-sets of $\Z_{2}^{4n}$ for $n=1,2,3$.
\begin{enumerate}
\item If $n=1$

\begin{eqnarray*}
\E_{4}(2)&=& \left\{ \G_{4}(2)\right\},\\
\mathcal{O}_{4}(2)&=& \left\{\G_{4}(1),\G_{4}(3)\right\} .
\end{eqnarray*}

\item If $n=2$

\begin{eqnarray*}
\E_{8}(4)&=& \left\{ \G_{8}(2),\G_{8}(4),\G_{8}(6)\right\}\\
\mathcal{O}_{8}(4)&=&\left\{ \G_{8}(3),\G_{8}(5)\right\} 
\end{eqnarray*}

\item If $n=3$

\begin{eqnarray*}
\E_{12}(6)&=&\left\{ \G_{12}(4),\G_{12}(6),\G_{12}(8)\right\}\\
\mathcal{O}_{12}(6)&=&\left\{ \G_{12}(3),\G_{12}(5),\G_{12}(7),\G_{12}(9)\right\}. 
\end{eqnarray*}

\end{enumerate}
\end{example}

\section{Schur ring with circulant basic sets}

In this section Schur ring with circulant basic sets are studied. Also, we define the free and 
non-free circulant $S$-sets, the symmetric, non-symmetric and antisymmetric circulant $S$-sets. 
We prove that all this $S$-sets are invariants under decimation.

\subsection{Circulant basic sets}

Let $C$ denote the cyclic permutation on the components $+$ and $-$ of $X$ in 
$\Z_{2}^{n}$ such that
\begin{equation}\label{cir1}
C(X)=C\left( x_{0},x_{1},...,x_{n-2},x_{n-1}\right) =\left(x_{1},x_{2},x_{3},...,x_{0}\right),
\end{equation}
that is, $C(x_{i})=x_{(i+1) mod n}$. The permutation $C$ is a generator of cyclic group 
$C_{n}=\left\langle C\right\rangle$ of order $n$. Let 
$X_{C}=Orb_{C_{n}}X=\{C^{i}(X):C^{i}\in C_{n} \}$. Therefore, $C_{n}$ defines a partition in
equivalent class on $\Z_{2}^{n}$ which is an $S$-partition and this we shall denote by 
$\Z_{2C}^{n}=\mathfrak{S}(\Z_{2}^{n},C_{n})$. It is worth mentioning that this Schur ring 
corresponds to the orbit Schur ring induced by the cyclic permutation automorphic subgroup 
$C_{n}\le S_n\le Aut(\Z_2^n)$. Likewise, it is worth noting that the Schur ring 
$\mathfrak{S}(\Z_{2}^{n},S_{n})$ is an orbit Schur ring induced by the permutation automorphic
subgroup $S_n\le Aut(\Z_2^n)$.

On the other hand, in general $\left\vert X_{C}\right\vert\neq n$, $X_{C} \in \Z_{2C}^{n}$. 
For example

\begin{equation}\label{cir2}
X=\left(+----+----+----\right) \in \G_{15}(3)\subset \Z_{2}^{15}
\end{equation}

has orbit size 5, not 15.

Now, let $X_{C},Y_{C}\in\Z_{2C}^{n}$, $X\neq Y$, such that $\lvert X_{C}\rvert = d_{1}$ and 
$\lvert Y_{C}\rvert = d_{2}$, $d_{1},d_{2}\mid n$. It is easy to see that 
$C_{n}$ defines a partition on $X_{C}Y_{C}$. Therefore, there are $Z_{iC}$ such that

\begin{equation}
X_{C}Y_{C}= \bigcup_{i=0}^{M-1}Z_{iC}
\end{equation}

and $\mid Z_{iC}\mid = N$, where $M=min\{d_{1},d_{2}\}$ and $N=max\{d_{1},d_{2}\}$.
When $X=Y$, we make $X_{C}X_{C}=X_{C}^{2}$ and $Z_{0C}=1_{C}$. The $S$-partition $\Z_{2C}^{n}$ 
defines a Schur ring where each $X_{C}$ in $\Z_{2C}^{n}$ will be called 
\textbf{circulant basic set}. An $S$-set de $\Z_{2C}^{n}$ will be called circulant $S$-set. 
Also it is easy to see that

\begin{eqnarray*}
1_{C}X_{C}&=&X_{C},\\
X_{C}Y_{C}&=&Y_{C}X_{C},\\
X_{C}(Y_{C}Z_{C})&=&(X_{C}Y_{C})Z_{C}
\end{eqnarray*}
with $X_{C},Y_{C},Z_{C}\in\Z_{2C}^{n}$.\\

On the other hand, let 
\begin{equation}
F_{d}(\Z_{2}^{n})=\bigcup_{\vert X_{C}\vert=d}X.
\end{equation}
Clearly $d$ divides to $n$ and the $X\in F_{d}(\Z_{2}^{d})$ have the form $X=(Y,Y,...,Y)$, with 
$Y\in\Z_{2}^{d}$. Then $F_{d}(\Z_{2C}^{n})=\bigcup_{\vert X_{C}\vert=d}X_{C}$ is an $S$-set of
$\Z_{2C}^{n}$, for each $d\vert n$. When $d=n$, we will to say that $C_{n}$ acts freely on 
$X_{C}$ and we denote $F_{n}(\Z_{2C}^{n})$ as $F(\Z_{2C}^{n})$. When $d<n$, we will to say that
$C_{n}$ don't act freely on $X_{C}$ and let $\widehat{F}(\Z_{2C}^{n})$ denote the set of the 
$X_{C}$ which are not frees under the action of $C_{n}$, namely

\begin{equation}\label{Z_not_libre}
\widehat{F}(\Z_{2C}^{n})=\bigcup_{d\mid n,d<n}F_{d}(\Z_{2C}^{n}).
\end{equation}

Therefore, 
\begin{eqnarray}\label{Z_libre}
\Z_{2C}^{n}&=&F(\Z_{2C}^{n})\cup \widehat{F}(\Z_{2C}^{n})\nonumber\\
&=&\bigcup_{d\mid n}F_{d}(\Z_{2C}^{n}). 
\end{eqnarray}

Take $d$ a divisor of $n$. We can see that $\bigcup_{r\mid d}F_{r}(\Z_{2C}^{n})$ is an
$S$-subgroup of the $S$-ring $\mathfrak{S}(\Z_{2}^{n},C_{n})$. 

On the other hand, when $n=p$ is an odd prime number 
$\left\vert\G_{p}(a)\right\vert$ is divisible by $p$, $0< a< p$, therefore 
$\left\vert X_{C}\right\vert =p$ for all $X_{C}\in\G_{p}(a)$ and 
$\widehat{F}(\Z_{2C}^{p})=\{\G_{p}(0),\G_{p}(p)\}$.

Now, we define $\G_{d}^{(n/d)}(a)=\G_{d}(a)\times \cdots \times \G_{d}(a)$, $n/d$ times. Then

\begin{equation}\label{cir4}
\G_{n}(an/d)\supset \G_{d}^{(n/d)}(a).
\end{equation}

and

\begin{equation}
\bigcup_{a=1}^{d-1}\G_{n}(an/d)\supset\bigcup_{a=1}^{d-1}\G_{d}^{(n/d)}(a)\supset F_{d}(\Z_{2}^{n}).
\end{equation}

If $n=p^{m}$ and if $s=rp^{k}, p\nmid r$,
\begin{eqnarray}\label{serie}
\G_{p^{m}}(s)&=&\G_{p^{m}}( rp^{k})\nonumber\\
&\supset& \G_{p^{m-1}}^{(p)}(rp^{k-1})\nonumber\\
&\supset&\G_{p^{m-2}}^{(p^{2})}(rp^{k-2})\\
&\vdots&\nonumber\\
&\supset&\G_{p^{m-k}}^{(p^{k})}(r)\nonumber
\end{eqnarray}

and

\begin{equation}
\bigcup_{r=1}^{p^{m-k}-1}\G_{p^{m-k}}^{(p^{k})}(r)\supset F_{p^{m-k}}(\Z_{2}^{p^{m}}).
\end{equation}

From (\ref{serie}) is followed that $\left\vert X_{C}\right\vert$ takes values in 
$p^{m},p^{m-1},\dots,p^{2},p,1$ for all $X_{C}\in \Z_{2C}^{p^{m}}$. 

In the following theorems we will show that a sufficient condiction for that 
$X_{C}^{2}\setminus\{1\}$ belong to $F(\Z_{2}^{n})$ is the parity of $n$

\begin{theorem}
Let $n$ be an even number. If $X_{C}\in F(\Z_{2C}^{n})$, then $X_{C}^{2}\setminus\{1\}\not\in F(\Z_{2C}^{n})$.
\end{theorem}
\begin{proof}
We make $XC^{n/2}(X)=Y_{n/2}$. Then $XC^{n/2}(X)=C^{n/2}(Y_{n/2})$ and $Y_{n/2}=C^{n/2}(Y_{n/2})$. 
Therefore $Y_{n/2}=(A,A)$ for some $A\in\Z_{2}^{n/2}$ and $\left\vert(XC^{n/2}(X))_{C}\right\vert$
has at most order $n/2$. Hence $X_{C}^{2}\not\in F(\Z_{2C}^{n})$.
\end{proof}

\begin{theorem}
Let $n$ be an odd number. If $X_{C}\in F(\Z_{2C}^{n})$, then $X_{C}^{2}\setminus\{1\}\in F(\Z_{2C}^{n})$.
\end{theorem}
\begin{proof}
Suppose $X_{C}^{2}\setminus\{1\}\not\in F(\Z_{2C}^{n})$. Then there is 
$a\in \left[1,n-1\right]$ such that $(XC^{a}X)_{C}\in F_{d}(\Z_{2C}^{n})$, with $d\mid n$. 
On the one hand, putting $Z=XC^{a}X$ we have $Z=C^{d}Z$ and $ZC^{d}Z=1$ where it follow 
that $Z=(A,A,...,A)$ with $A\in\Z_{2}^{d}$. On the other hand, putting $W=XC^{d}X$ we have 
$W=C^{a}W$, $WC^{a}W=1$ and $a\vert n$. Then $W=(B,B,...,B)$ with $B\in\Z_{2}^{a}$. Now, suppose 
that $a<d$ and let
$$X=(x_{0},x_{1},...,x_{a-1},x_{a},...,x_{d-1},x_{d},...,x_{n-1}).$$
Then $XC^{a}X=(A,A,...,A)$ implies that 
\begin{equation*}
\begin{array}{c}
x_{0}=x_{2a}=x_{4a}=\cdots\\
x_{1}=x_{2a+1}=x_{4a+1}=\cdots\\
\vdots\\
x_{a-1}=x_{3a-1}=x_{5a-1}=\cdots\\
x_{a}=x_{3a}=x_{5a}=\cdots\\
\vdots\\
x_{2a-1}=x_{4a-1}=x_{6a-1}=\cdots\\
\end{array}
\end{equation*}
Then $X=(Y,Y,...,Y)$ with $Y\in\Z_{2}^{2a}$. But this is impossible because $n$ is an odd number.
Equally, $XC^{d}X=(B,B,...,B)$ implies that $X=(Y,Y,...,Y)$ with $Y\in\Z_{2}^{2d}$, which is not
possible because $n$ is an odd number. Therefore $X_{C}^{2}\setminus\{1\}\not\in F(\Z_{2C}^{n})$
leads to a contradiction and $X_{C}^{2}\setminus\{1\}\in F(\Z_{2C}^{n})$.
\end{proof}

\subsection{Circulant $S$-Sets Invariant by Decimation}

Let $\delta_{k}\in S_{n-1}$ act on $X\in\Z_{2}^{n}$ by decimation, that is, 
$\delta_{k}(x_{i})=x_{ki(\mod n)}$ for all $x_{i}$ in $X$, $(k,n)=1$ and let $\Delta_{n}$ denote 
the set of this $\delta_{k}$. The set $\Delta_{n}$ is a group of order $\phi(n)$ isomorphic to
$\Z_{n}^{*}$, the group the units of $\Z_{n}$, where $\phi$ is called the Euler totient function. 
In this section another $S$-partitions on $\Z_{2}^{n}$ are constructed via the action of the group 
$\Delta_{n}$ and $\Delta_{n}C_{n}$. Also, circulant $S$-sets invariant by decimation are defined. 
We begin with the following $S$-partition

\begin{theorem}
$\mathfrak{S}(\Z_{2}^{n},\Delta_{n})$ is an $S$-partition.
\end{theorem}
\begin{proof}
Let $X_{\Delta_{n}}$ denote the orbit of $X\in\Z_{2}^{n}$ under the action of $\Delta_{n}$. It is 
too easy to see that $1_{\Delta_{n}}=1$, $X_{\Delta_{n}}^{-1}=X_{\Delta_{n}}$ and also that
$X_{\Delta_{n}}Y_{\Delta_{n}}=\bigcup_{\delta_{r}\in\Delta_{n}}(X\delta_{r}Y)_{\Delta_{n}}$.
\end{proof}

Now, we show that the $S$-sets $F_{d}(\Z_{2C}^{n})$ are invariants by decimation. With this
result we can then easily to obtain another $S$-partition on $\Z_{2}^{n}$.

\begin{theorem}
If $X_{C}\in F_{d}(\Z_{2C}^{n})$, then $(\delta_{r}X)_{C}\in F_{d}(\Z_{2C}^{n})$. 
\end{theorem}
\begin{proof}
Since $C_{n}$ doesn't act freely on $X$, then $X=(Y,Y,\dots,Y)\in F_{d}(\Z_{2}^{n})$ with 
$Y=(y_{0},y_{1},...,y_{d-1})\in \Z_{2}^{d}$, $d\mid n$. 
By the periodicity of $X$, $\delta_{r}$ maps $y_{i}\rightarrow y_{ri \Mod{d}}$, $0\leq i \leq d-1$, 
$\frac{n}{d}$ times.
\end{proof}

Let $G=\Delta_{n}C_{n}$ and let $X_{G}$ denote the orbit of $X$ under the action of $G$. Then

\begin{equation}
X_{G} = \bigcup_{r\in\Z_{n}^{*}}(\delta_{r}X)_{C}.
\end{equation}

If $Y\in X_{G}$, then $Y$ has the form 

\begin{equation}
Y=(x_{r0+j},x_{r1+j},...,x_{r(n-1)+j}).
\end{equation}

From the above theorem and following the equation (\ref{Z_libre}) we have the corollary

\begin{corollary}
$\mathfrak{S}(\Z_{2}^{n},\Delta_{n}C_{n})$ is an $S$-partition.
\end{corollary}
\begin{proof}
From the relation $C^{i}\delta_{r}=\delta_{r}C^{ir}$ it follows that 
$\delta_{r}X_{C}=(\delta_{r}X)_{C}$. Then, by using the above theorem and (\ref{Z_libre}) we
obtain to the desired result.
\end{proof}

On the other hand, we note by $RX$ the reversed sequence $RX = (x_{n-1},...,x_{1},x_{0})$. 
This permutation play an important role in the classification of $\Z_{2C}^{n}$.

\begin{definition}
Let $X_{C}\in \Z_{2C}^{n}$. We shall call $X_{C}$ symmetric if exists $Y\in X_{C}$ such that
$RY=Y$ and otherwise we say it is non symmetric. We make $Sym(\Z_{2C}^{n})$ the set of all $X_{C}$
symmetric and $\widehat{Sym}(\Z_{2C}^{n})$ the set of all $X_{C}$ nonsymmetric.
\end{definition}

Then the $S$-partition $\Z_{2C}^{n}$ can be expressed as 

\begin{equation}\label{Z_sym}
Z_{2C}^{n}=Sym(Z_{2C}^{n})\oplus\widehat{Sym}(Z_{2C}^{n}).
\end{equation}

From (\ref{Z_libre}) and (\ref{Z_sym}) it follows that

\begin{eqnarray}
F(\Z_{2C}^{n})&=&Sym(F(\Z_{2C}^{n}))\oplus\widehat{Sym}(F(\Z_{2C}^{n}))\\
\widehat{F}(\Z_{2C}^{n})&=&Sym(\widehat{F}(\Z_{2C}^{n}))\oplus\widehat{Sym}(\widehat{F}(\Z_{2C}^{n}))\\
Sym(\Z_{2C}^{n})&=&F(Sym(\Z_{2C}^{n}))\oplus\widehat{F}(Sym(\Z_{2C}^{n}))\\
\widehat{Sym}(\Z_{2C}^{n})&=&F(\widehat{Sym}(\Z_{2C}^{n}))\oplus\widehat{F}(\widehat{Sym}(\Z_{2C}^{n})).
\end{eqnarray}

Let $SF=Sym(F(\Z_{2C}^{n}))$, $\widehat{S}F=\widehat{Sym}(F(\Z_{2C}^{n}))$, 
$S\widehat{F}=Sym(\widehat{F}(\Z_{2C}^{n}))$, $\widehat{SF}=\widehat{Sym}(\widehat{F}(\Z_{2C}^{n}))$.
Then

\begin{equation}
\Z_{2C}^{n} = SF\oplus\widehat{S}F\oplus S\widehat{F}\oplus\widehat{SF}.
\end{equation}

There are three commutation relations among $\delta_{r}$, $C$ and $R$:
\begin{eqnarray}
\delta_{r}R&=&R\delta_{r}C^{r-1},\label{R_d_C}\\
RC&=&C^{-1}R,\label{R_C}\\
C^{i}\delta_{r}&=&\delta_{r}C^{ir}\label{C_d}
\end{eqnarray}

Then, there are another $S$-partitions on $\Z_{2}^{n}$, namely $\mathfrak{S}(\Z_{2}^{n},H_{n})$, 
$\mathfrak{S}(\Z_{2}^{n},H_{n}C_{n})$ and $\mathfrak{S}(\Z_{2}^{n},H_{n}\Delta_{n}C_{n})$, where 
$H_{n}=\{e,R\}$ is the reversing automorphic subgroup in $Aut(\Z_{2}^{n})$. Therefore, 
$Sym(\Z_{2C}^{n})\subset\mathfrak{S}(\Z_{2}^{n},H_{n}C_{n})$.
 
On the other hand, there is a 1-1 correspondence between set of basic sets of an $S$-partition and
$S$-sets of $S$-ring through to make
\begin{equation*}
\{T_{i_{1}},T_{i_{2}},...,T_{i_{k}}\}\mapsto\bigcup_{r=1}^{k}T_{i_{r}}.
\end{equation*}


Then, we can to see the set $SF$, $\widehat{S}F$, $S\widehat{F}$, $\widehat{SF}$, $Sym(\Z_{2C}^{n})$
and $\widehat{Sym}(\Z_{2C}^{n})$ as $S$-sets of $\mathfrak{S}(\Z_{2}^{n},C_{n})$ and we shall show
that this are invariant by decimation

\begin{theorem}\label{theo_decimation_partition}
The $S$-set $Sym(\Z_{2C}^{n})$ is invariant under the action of $\Delta_{n}$.
\end{theorem}
\begin{proof}
From (\ref{R_d_C}), (\ref{R_C}) and (\ref{C_d}) we have $\delta_{r}R=C^{r^{-1}-1}R\delta_{r}$. Then,
we taking $X_{C}$ in $Sym(\Z_{2C}^{n})$ with $RX=X$ we have
\begin{equation*}
R(\delta_{r}X)_{C}=(R\delta_{r}X)_{C}=(C^{1-r^{-1}}\delta_{r}RX)_{C}=(C^{1-r^{-1}}\delta_{r}X)_{C}=
(\delta_{r}X)_{C}
\end{equation*}
for all $\delta_{r}\in\Delta_{n}$.
\end{proof}

\begin{corollary}\label{cor_decimation_partition}
The $S$-set $\widehat{Sym}(\Z_{2C}^{n})$ is invariant under the action of $\Delta_{n}$.
\end{corollary}
\begin{proof}
Follows from (\ref{Z_libre}) and (\ref{Z_sym}).
\end{proof}

\begin{corollary}
The $S$-sets $SF$, $S\widehat{F}$, $\widehat{S}F$ and $\widehat{SF}$ are invariant under the 
action of $\Delta_{n}$.
\end{corollary}
\begin{proof}
From (4.18)-(4.21) and from theorem \ref{theo_decimation_partition} and corollary \ref{cor_decimation_partition}.
\end{proof}

Finally, we define the $S$-set antisymmetric in $\Z_{2C}^{n}$ and we show that this is invariant
under the action of $\Delta_{n}$

\begin{definition}
A set $X_{C}$ in $\Z_{2C}^{n}$ is antisymmetric if exist $Y$ in $X_{C}$ such that $RY=-Y$. We
shall denote to the antisymmetric sets in $\Z_{2C}^{n}$ with $ASym(\Z_{2C}^{n})$.
\end{definition}

As $RX_{C}=-X_{C}$ for all $X_{C}\in ASym(\Z_{2C}^{n})$, then $R(X_{C}Y_{C})=X_{C}Y_{C}$ and
$ASym^{2}(\Z_{2C}^{n})=ASym(\Z_{2C}^{n})ASym(\Z_{2C}^{n})=Sym(\Z_{2C}^{n})$.

\begin{theorem}
The $S$-set $ASym(\Z_{2C}^{n})$ is invariant under the action of $\Delta_{n}$.
\end{theorem}
\begin{proof}
Equal to proof of the theorem \ref{theo_decimation_partition}.
\end{proof}

\subsection{Periodic autocorrelation function}

Let $X=\{x_{i}\}$ and $Y=\{y_{i}\}$ be two complex-valued sequences of period $n$. The periodic
correlation of $X$ and $Y$ at shift $k$ is the product defined by:

\begin{equation}
\mathsf{P}_{X,Y}(k)=\sum\limits_{i=0}^{n-1}x_{i}\overline{y}_{i+k},\ k=0,1,...,n-1,
\end{equation}

where $\overline{a}$ denotes the complex conjugation of $a$ and $i+k$ is calculated modulo $n$.
If $Y=X$, the correlation $\mathsf{P}_{X,Y}(k)$ is denoted by $\mathsf{P}_{X}(k)$ and is the
autocorrelation of $X$. Obviously, 

\begin{eqnarray}
\mathsf{P}_{X}(k)&=&\overline{\mathsf{P}_{X}(n-k)},\\
\mathsf{P}_{RX}(k)&=&\mathsf{P}_{X}(k),\label{reversing}\\
\mathsf{P}_{-X}(k)&=&\mathsf{P}_{X}(k),\label{negative}\\
\mathsf{P}_{C^{i}X}(k)&=&\mathsf{P}_{X}(k)\label{cyclotomic_autoc},
\end{eqnarray}
for all $0\leq i\leq n-1$ and for all $X$ in $\Z_{2}^{n}$.\\

If $X$ is a $\Z_{2}$-sequence of length $n$, $\mathsf{P}_{X}(k)= 2\omega \left\{Y_{k}\right\}-n$,
where $ Y_{k}=XC^{k}X $. Also by (\ref{producto}), if $X\in \G_{n}(a)$, then

\begin{equation}
\mathsf{P}_{X}(k)=n-4a+4i_{k},
\end{equation} 

for some $0\leq i_{k}\leq a$ and $n-\mathsf{P}_{X}(k)$ is divisible by 4 for all $k$ (see [11]).\\ 

In the following propositions we show the relationship existing between the Hamming weight $\omega$
and the reversing map $R$

\begin{proposition}
Let $X,Y\in \Z_{2}^{n}$. Then
\begin{eqnarray}
\omega(X)&=&\omega(RX).\\
\omega(XRY)&=&\omega(YRX).\label{corollarySym}
\end{eqnarray}
\end{proposition}
\begin{proof}
Clearly $X$ and $RX$ have the same Hamming weight. As $R(XRY)=YRX$ the statement follows.
\end{proof}

\begin{proposition}
\begin{enumerate}
\item Let $X\in \Z_{2}^{n}$. If $n$ is an odd number, then
\begin{equation}
\omega(XRX)=1 + 2\omega(BRD), 
\end{equation}
with $X=(B,x,D)$, $B,D\in \Z_{2}^{(n-1)/2}$.
\item If $n$ is an even number, then
\begin{equation}
\omega(XRX)=2\omega(BRD),
\end{equation}
with $X=(B,D)$, $B,D\in \Z_{2}^{n/2}$.
\end{enumerate}
\end{proposition}
\begin{proof}
Suppose $n$ an odd number. As $XRX$ is symmetric, then $X=(B,x,D)$, $x\in\{\pm\}$ and 
$B,D\in \Z_{2}^{(n-1)/2}$ and $XRX=(BRD,+,DRB)$. First statement is followed. The equation second is 
proved in way similar.
\end{proof}

On the other hand, if $X\in \G_{n}(a)$ and $Y\in\G_{n}(b)$, then

\begin{eqnarray}\label{ecplano}
\sum\limits_{j=0}^{n-1}\mathsf{P}_{X,Y}(k)&=&\sum\limits_{k=0}^{n-1}\sum\limits_{i=0}^{n-1}x_{i}y_{i+k}\nonumber \\
&=&\sum\limits_{i=0}^{n-1}\sum\limits_{k=0}^{n-1}x_{i}y_{i+k}\nonumber \\
&=&\sum\limits_{i=0}^{n-1}x_{i}\sum\limits_{k=0}^{n-1}y_{i+k}\nonumber \\
&=&(2a-n)(2b-n).
\end{eqnarray}

Now, let 
$$(\mathsf{P}_{X}(0),\mathsf{P}_{X}(1),...,\mathsf{P}_{X}(n-1))$$
denote the autocorrelation vector of $X_{C}$ in $\Z_{2C}^{n}$ and let $\mathfrak{A}(\Z_{2C}^{n})$ denote the set of all this. Let $X_{1}+X_{2}+\cdots +X_{n}=d$ denote the plane in $\Z^{n}$ in the 
indeterminates $X_{i}$, $i=1,2,...,n$ and let 
$\theta:\Z_{2C}^{n}\rightarrow\mathfrak{A}(\Z_{2C}^{n})$ be the map defined by 
$\theta(X_{C})=(\mathsf{P}_{X}(0),\mathsf{P}_{X}(1),\dots ,\mathsf{P}_{X}(n-1))$, where $\theta$
is defined by a representative of $X_{C}$. Therefore by (\ref{cyclotomic_autoc}), $\theta$ is 
well-defined. Then from (\ref{ecplano}), $\theta$ sends the plane $\G_{n}(a)$ in the plane 
$X_{1}+X_{2}+\cdots +X_{n}=(2a-n)^{2}$.

On the other hand, the decimation group $\Delta_{n}$ do not alter the set of values which 
$\mathsf{P}_{X}(k)$ takes on, but merely the order in which they appear. Below we prove what 
was stated above.

Let $\delta_{r}X = Y = (y_{0},y_{1},...,y_{n-1})$. Then
\begin{equation}\label{auto_decimated}
\mathsf{P}_{Y}(i)= \sum_{k=0}^{n-1}y_{k}\overline{y_{k+i}} = \sum_{k=0}^{n-1}x_{rk}\overline{x_{r(k-i)}} = \sum_{rk=0}^{n-1}x_{rk}\overline{x_{rk + ri}} = \mathsf{P}_{X}(ri).
\end{equation}
Hence $XC^{i}X\longrightarrow XC^{ri}X$ and $\delta_{r}$ is a permutation over $\theta(X_{C})$.

From above we have the commutative diagram
\begin{equation}\label{diagram}
\xymatrix{
 \Z_{2C}^{n} \ar[d]^{\theta} \ar[r]^{\delta_{r}} & \Z_{2C}^{n} \ar[d]^{\theta}\\
   \mathfrak{A}(\Z_{2C}^{n}) \ar[r]^{\delta_{r}} & \mathfrak{A}(\Z_{2C}^{n}) 
}
\end{equation}
and $\theta \circ \delta_{r} = \delta_{r}\circ \theta.$\\
Let $\Delta_{n}(\theta(X_{C}))$ denote the set
\begin{equation}
\{\delta_{r}(\theta (X_{C})):\delta_{r}\in \Delta_{n} \}.
\end{equation}
Then $\theta$ is a mapping of equivalence class, thus $\theta: \Delta_{n}(X_{C})\rightarrow \Delta_{n}(\theta(X_{C}))$. Now we show that in general it is hold that

\begin{proposition}
Let $X_{C}\in\Z_{2C}^{n}$. Then
\begin{equation}
\theta(X_{C})=\theta((-X)_{C})=\theta((RX)_{C})
\end{equation}
\end{proposition}
\begin{proof}
Is followed from (\ref{reversing}) and (\ref{negative}).
\end{proof}

Such as we shall see in the section following there are $X_{C}$ in $\Z_{2C}^{n}$ such that
$\theta(X_{C})=(n,a,a,...,a)$. These $X_{C}$ hold that $\theta(\delta_{r}X_{C})=\theta(X_{C})$ for
all $\delta_{r}$ in $\Delta_{n}$. Then there is $Y$ in $X_{C}$ such that $\delta_{r}Y=Y$ for
some $\delta_{r}\in\Delta_{n}$. Thus $Y$ is fixed by $\delta_{r}$. Now, we shall define the 
$S$-set of the all $X_{C}$ that are fixed by $\delta_{r}$.

\begin{definition}
Take $X_{C}\in\Z_{2C}^{n}$. The orbit $X_{C}$ is $\delta_{r}$-invariant if exist $Y$ in $X_{C}$ 
such that $\delta_{r}Y=Y$ for some $\delta_{r}$ in $\Delta_{n}$ and let $\mathbb{I}_{nC}(r)$ 
denote the $S$-set of $\delta_{r}$-invariant orbits of $\Z_{2C}^{n}$. 
\end{definition}

\begin{theorem}
$\Delta_{n}$ defines a partition on $\mathbb{I}_{nC}(r)$.
\end{theorem}
\begin{proof}
From previous section we know that $\Delta_{n}$ defines a partition on $\Z_{2C}^{n}$. Therefore 
the statement is true for $r=1$. Now, take $X_{C}$ in $\mathbb{I}_{nC}(r)$, $r\neq1$, and suppose 
$\delta_{r}X=X$. It is enough with to take a $\delta_{s}\neq\delta_{r}$ in $\Delta_{n}$, with 
$s\neq1$. From the relation $C^{i}\delta_{s}=\delta_{s}C^{si}$ it is follows that 
$\delta_{r}(\delta_{s}X)_{C}=(\delta_{s}X)_{C}$. Then $(\delta_{s}X)_{C}\in\mathbb{I}_{nC}(r)$.
\end{proof}

Then, from previous theorem is followed that

\begin{corollary}
$\mathbb{I}_{n}(r)$ is an $S$-subgroup of $\mathfrak{S}(\Z_{2}^{n},\Delta_{n})$ for each $\delta_{r}$
in $\Delta_{n}$.
\end{corollary}

\section{Hadamard Matrices}
A Hadamard matrix $H$ is a $n$ by $n$ matrix all of whose entries are $+1$ or 
$-1$ which satisfies $HH^{t}=nI_{n}$, where $H^{t}$ is the transpose of $H$
and $I_{n}$ is the unit matrix of order $n$. It is also known that, if a
Hadamard matrix of order $n>1$ exists, $n$ must have the value $2$ or be
divisible by 4. It has been conjecture that this condition also insures the
existence of a Hadamard matrix.

Two Hadamard matrices $H$ and $H^{\prime}$ are equivalents if one can be obtained from the other
by perfoming a finite sequence of the following operations:
\begin{enumerate}
\item permute the rows or the columns,
\item multiply a row or a column by $-1$.
\end{enumerate}

A important result in this paper is to prove that if a Hadamard matrix exists then this or its 
equivalente matrix must be contained in a complete maximal $S$-set. 

\begin{theorem}
If $H$ is a Hadamard matrix, then this or its equivalent matrix $H^{\prime}$ there exist either in
$\E_{4n}(2n)$ or in $\mathcal{O}_{4n}(2n)$. 
\end{theorem}
\begin{proof}
Let $H_{i},H_{j}$ be rows vector in $H$. As $(H_{i},H_{j})=0$, then $H_{i}H_{j}\in\G_{4n}(2n)$
if $i\neq j$ and $H_{i}H_{j}\in\G_{4n}(4n)$ if $i=j$. If $H$ is contained either in 
$\E_{4n}(2n)$ or in $\mathcal{O}_{4n}(2n)$, then nothing should be proved.
Therefore suppose that $H$ is not contained in some $\G_{4n}(2n)$-complete $S$-set. By multiplying
the columns of $H$ by appropiate signs we can obtain an equivalent matrix $H^{\prime}$ whose firt 
row be a $\Z_{2}$-sequence in some basic set of either $\E_{4n}(2n)$ or 
$\mathcal{O}_{4n}(2n)$. Then, by definition of $\G_{4n}(2n)$-complete $S$-set all row 
of $H^{\prime}$ belongs to either $\E_{4n}(2n)$ or $\mathcal{O}_{4n}(2n)$.
\end{proof}

Next,we shall show an example illustrating this

\begin{example}
Let
\begingroup\makeatletter\def\f@size{5}\check@mathfonts
\begin{equation*}
H=\left(
\begin{array}{cccccccccccc}
+&-&-&-&-&-&-&-&-&-&-&-\\
+&+&-&+&-&-&-&+&+&+&-&+\\
+&+&+&-&+&-&-&-&+&+&+&-\\
+&-&+&+&-&+&-&-&-&+&+&+\\
+&+&-&+&+&-&+&-&-&-&+&+\\
+&+&+&-&+&+&-&+&-&-&-&+\\
+&+&+&+&-&+&+&-&+&-&-&-\\
+&-&+&+&+&-&+&+&-&+&-&-\\
+&-&-&+&+&+&-&+&+&-&+&-\\
+&-&-&-&+&+&+&-&+&+&-&+\\
+&+&-&-&-&+&+&+&-&+&+&-\\
+&-&+&-&-&-&+&+&+&-&+&+
\end{array}
\right)
\end{equation*}
\endgroup
a Hadamard matrix in $\G_{12}(1)\cup\G_{12}(7)$. Then the equivalent matrix 
\begingroup\makeatletter\def\f@size{5}\check@mathfonts
\begin{equation*}
H^{\prime}=\left(
\begin{array}{cccccccccccc}
+&-&+&+&-&-&-&+&-&+&-&+\\
+&+&+&-&-&-&-&-&+&-&-&-\\
+&+&-&+&+&-&-&+&+&-&+&+\\
+&-&-&-&-&+&-&+&-&-&+&-\\
+&+&+&-&+&-&+&+&-&+&+&-\\
+&+&-&+&+&+&-&-&-&+&-&-\\
+&+&-&-&-&+&+&+&+&+&-&+\\
+&-&-&-&+&-&+&-&-&-&-&+\\
+&-&+&-&+&+&-&-&+&+&+&+\\
+&-&+&+&+&+&+&+&+&-&-&-\\
+&+&+&+&-&+&+&-&-&-&+&+\\
+&-&-&+&-&-&+&-&+&+&+&-
\end{array}
\right)
\end{equation*}
\endgroup
is contained in $\E_{12}(6)=\{\G_{12}(4),\G_{12}(6),\G_{12}(8)\}$.
\end{example}

In the following sections only two types of Hadamard matrices are studied: circulant and with one
core.

\subsection{Circulant Hadamard matrices}

A circulant Hadamard matrix of order $n$ is a square matrix of the form 
\begin{equation}
H =
\left(\begin{array}{cccc}
	a_{1} & a_{2} & \cdots & a_{n}\\
	a_{n} & a_{1} & \cdots & a_{n-1}\\
	\cdots & \cdots & \cdots & \cdots\\
	a_{2} & a_{3} & \cdots & a_{1}
\end{array}\right)
\end{equation}

No circulant Hadamard matrix of order larger than 4 has ever been found. Then we have the following

\begin{conjecture}
No circulant Hadamard matrix of order larger than 4 exists.
\end{conjecture}

To prove this conjecture is equivalent to prove that there is no $X_{C}$ such that 
$\theta(X_{C})=(4n,0,...,0)$. Thus, it is enough to prove that $\mathsf{P}_{X}(k)\neq 0$ for 
some $k\neq 0$. Then we prove that circulant Hadamard matrices cant'n to exist if these have some
special structure. 

\begin{definition}
We shall say that a binary sequence $X$ in $\Z_{2}^{nr}$ is \textbf{partitioned} in a basic set 
of the Schur ring $\mathfrak{S}(\Z_{2}^{n},G)$, with $G\leq Aut(Z_{2}^{n})$, if
$X=(Y_{1},Y_{2},...,Y_{r})$ where $Y_{i}\in A_{G}$ and $A_{G}$ is the orbit of $A$ under the 
action of $G$.
\end{definition}

We will prove that a circulant Hadamard matrix never is partitioned in 
$\mathfrak{S}(\Z_{2}^{n},S_{n})$

\begin{theorem}\label{theo_circ_had_1}
There is no circulant Hadamard matrices $X_{C}$ in $\Z_{2C}^{4rn}$ with 
$$X=(Y_{1},Y_{2},...,Y_{2r})\in\G_{2n}(a)\times\cdots\times\G_{2n}(a)\subset\G_{4nr}(2ar),$$ 
$Y_{i}\in\G_{2n}(a)$ for all $i=1,2,...,2r$ if 
\begin{enumerate}
\item $nr$ is an odd number and $\sum_{i=1}^{2r}(\frac{1}{2}\omega(Y_{i}Y_{i+1})+2a-2n)$ an
even number,
\item or if $nr$ is an even number and $\sum_{i=1}^{2r}(\frac{1}{2}\omega(Y_{i}Y_{i+1})+2a-2n)$ an
odd number.
\end{enumerate}
\end{theorem}
\begin{proof}
Take $X=(Y_{1},Y_{2},...,Y_{2r})$ in $\Z_{2}^{4nr}$. Then
\begin{eqnarray*}
\omega(XC^{2n}X)&=&\omega(Y_{1}Y_{2},Y_{2}Y_{3},...,Y_{2r}Y_{1}))\\
&=&\omega(Y_{1}Y_{2})+\omega(Y_{2}Y_{3})+\cdots+\omega(Y_{2r}Y_{1}).
\end{eqnarray*}
Suppose that $\omega(XC^{2n}X)=2rn$. As the $Y_{i}Y_{i+1}\in\G_{2n}(a)^{2}$ it follows that
\begin{equation*}
4rn-4ra+2\sum_{i=1}^{2r}h_{i}=2rn
\end{equation*}
where the $h_{i}$ are integer ranging in $[0,a]$. Thereupon

\begin{equation*}
ar=\frac{nr}{2}+\frac{1}{2}\sum_{i=1}^{2r}h_{i}
\end{equation*}
is an integer only if $nr$ and $\sum_{i=1}^{2r}h_{i}$ both are even numbers or both are odd numbers.
\end{proof}

As a consequence we have the following result on circulant Hadamard matrices partitioned in basic
sets of the Schur rings $\mathfrak{S}(\Z_{2}^{2n},H_{2n}\Delta_{2n}C_{2n})$, 
$\mathfrak{S}(\Z_{2}^{2n},H_{2n}C_{2n})$, $\mathfrak{S}(\Z_{2}^{2n},\Delta_{2n}C_{2n})$,
$\mathfrak{S}(\Z_{2}^{2n},C_{2n})$, $\mathfrak{S}(\Z_{2}^{2n},\Delta_{2n})$ and 
$\mathfrak{S}(\Z_{2}^{2n},H_{2n})$.

\begin{corollary}
Let $A_{G}$ be a basic set in $\mathfrak{S}(\Z_{2}^{2n},G)$, where $G$ is some of the following
groups $H_{2n}\Delta_{2n}C_{2n}$, $\Delta_{2n}C_{2n}$, $H_{2n}C_{2n}$, $C_{2n}$, $\Delta_{2n}$. Then there is no circulant Hadamard matrices $X_{C}$ in $\Z_{2C}^{4rn}$ with 
$$X=(Y_{1},Y_{2},...,Y_{2r})\in A_{G}\times\cdots\times A_{G}\subset\G_{4nr}(2ar),$$ 
$Y_{i}\in A_{G}$ for all $i=1,2,...,2r$ if 
\begin{enumerate}
\item $nr$ is an odd number and $\sum_{i=1}^{2r}(\frac{1}{2}\omega(Y_{i}Y_{i+1})+2a-2n)$ an
even number,
\item or if $nr$ is an even number and $\sum_{i=1}^{2r}(\frac{1}{2}\omega(Y_{i}Y_{i+1})+2a-2n)$ an
odd number.
\end{enumerate}
\end{corollary}

When $G=H_{2n}$, then $X=(A,RA)$ is a symmetric sequence. The following corollary corresponds to
symmetric circulant Hadamard matrices

\begin{corollary}\label{circ_had_sym}
Let $X_{C}\in\Z_{2C}^{4n}$, $n$ an odd number, with $X=(A,RA)$. Then $X_{C}$ is not Hadamard.
\end{corollary}

Finally, we show another one structure which is not a circulant Hadamard matrix

\begin{definition}
We shall say that a binary sequence $X$ in $\Z_{2}^{2nr}$ is \textbf{partitioned with alternated sign} 
in a basic set of the Schur ring $\mathfrak{S}(\Z_{2}^{n},G)$, with $G\leq Aut(Z_{2}^{n})$, 
if $X=(Y_{1},-Y_{2},...,Y_{2r-1},-Y_{2r})$ where $Y_{i}\in A_{G}$ and $A_{G}$ is the orbit of 
$A$ under the action of $G$.
\end{definition}

We will prove that a circulant Hadamard matrix never is partitioned with alternated sign in 
$\mathfrak{S}(\Z_{2}^{n},S_{n})$

\begin{theorem}\label{theo_circ_had_2}
There is no circulant Hadamard matrices $X_{C}\in\Z_{2C}^{4n}$, $n$ an odd number, with 
\begin{eqnarray*}
X&=&(Y_{1},-Y_{2},...,Y_{2r-1},-Y_{2r})\\
&\in&\G_{2n}(a)\times\G_{2n}(2n-a)\times\cdots\times\G_{2n}(a)\times\G_{2n}(2n-a)\subset\G_{4rn}(2rn),
\end{eqnarray*}
with $Y_{i}\in\G_{2n}(a)$, $i=1,2,...,2r$ if
\begin{enumerate}
\item $nr$ is an odd number and $\sum_{i=1}^{2r}(-\frac{1}{2}\omega(-Y_{i}Y_{i+1})+2a)$ an
even number,
\item or if $nr$ is an even number and $\sum_{i=1}^{2r}(-\frac{1}{2}\omega(-Y_{i}Y_{i+1})+2a)$ an
odd number.
\end{enumerate}
\end{theorem}
\begin{proof}
Take $X=(Y_{1},-Y_{2},...,-Y_{2r})$ in $\Z_{2}^{4nr}$, with $Y_{i}\in\G_{2n}(a)$. Then
\begin{eqnarray*}
\omega(XC^{2n}X)&=&\omega(-Y_{1}Y_{2},-Y_{2}Y_{3},...,-Y_{2r}Y_{1}))\\
&=&\omega(-Y_{1}Y_{2})+\omega(-Y_{2}Y_{3})+\cdots+\omega(-Y_{2r}Y_{1})\\
&=&2r(2n-(2n-2a))-2\sum_{i=1}^{2r}h_{i}\\
&=&4ar-2\sum_{i=1}^{2r}h_{i}
\end{eqnarray*}
where the $h_{i}$ are integer ranging in $[0,a]$. Suppose that $\omega(XC^{2n}X)=2rn$. Then it 
is followed that

\begin{equation*}
ar=\frac{nr}{2}+\frac{1}{2}\sum_{i=1}^{2r}h_{i}
\end{equation*}
is an integer only if $nr$ and $\sum_{i=1}^{2r}h_{i}$ both are even numbers or both are odd numbers.
\end{proof}

Equally, we have the following corollary for the non-existence of circulant Hadamard matrices 
partitioned with alternated sign in basic sets of $\mathfrak{S}(\Z_{2}^{2n},G)$

\begin{corollary}
Let $A_{G}$ be a basic set in $\mathfrak{S}(\Z_{2}^{2n},G)$, where $G$ is some of the following
groups $H_{2n}\Delta_{2n}C_{2n}$, $\Delta_{2n}C_{2n}$, $H_{2n}C_{2n}$, $C_{2n}$, $\Delta_{2n}$. Then there is no circulant Hadamard matrices $X_{C}$ in $\Z_{2C}^{4rn}$ with 
$$X=(Y_{1},-Y_{2},...,Y_{2r-1},-Y_{2r})\in A_{G}\times\cdots\times A_{G}\subset\G_{4nr}(2nr),$$ 
$Y_{i}\in A_{G}\subset\G_{2n}(a)$ for all $i=1,2,...,2r$ if 
\begin{enumerate}
\item $nr$ is an odd number and $\sum_{i=1}^{2r}(-\frac{1}{2}\omega(-Y_{i}Y_{i+1})+2a)$ an
even number,
\item or if $nr$ is an even number and $\sum_{i=1}^{2r}(-\frac{1}{2}\omega(-Y_{i}Y_{i+1})+2a)$ an
odd number.
\end{enumerate}
\end{corollary}

Equally, when $G=H_{2n}$, then $X=(A,-RA)$ is a antisymmetric sequence. The following corollary
corresponds to antisymmetric circulant Hadamard matrices

\begin{corollary}\label{circ_had_antisym}
Let $X_{C}\in\Z_{2C}^{4n}$, $n$ an odd number, with $X=(A,-RA)$. Then $X_{C}$ is not Hadamard.
\end{corollary}

The corollary \ref{circ_had_sym} is a known theorem proved in the 1965 paper [12](Corollary 2).
In the theorems \ref{theo_circ_had_1} and \ref{theo_circ_had_2} stronger results are proved.

\subsection{Hadamard matrices with one circulant core}

A Hadamard matrix with one circulant core of order $p$ is a $p\times p$ matrix of the form
\begin{equation*}
H=
\left( 
\begin{array}{cc}
1 & e \\ 
e^{t} & A_{C}
\end{array}
\right)
\end{equation*}
where $e$ is the row vector $(1,1,1,...,1)$ of dimension $p$ and $e^{t}$
the transposed vector of $e$ and $A_{C}=(a_{i,j})$ a circulant matrix or circulant core of order
$n-1$. A Hadamard matrix of order $p+1$ with circulant core can be constructed if
\begin{enumerate}
\item[(1)] $p\equiv3\mod4$ is a prime
\item[(2)] $p=q(q+2)$ where $q$ and $q+2$ are both primes
\item[(3)] $p=2^{t}-1$\ where $t$ is a positive integer
\item[(4)] $p=4x^{2}+27$ where $p$ is a prime and $x$ a positive integer.
\end{enumerate}

We have the following

\begin{conjecture}
Above are the only possible orders for an Hadamard matrix with one circulant core.
\end{conjecture}

As $H$ above is Hadamard, then $A_{C}$ is in $\G_{pC}\left(\frac{p-1}{2}\right)$. Also,
the autocorrelation vector $\theta(A_{C}^{2})$ is equal to $(p,-1,-1,...,-1)$. Hence 
$A_{C}\in\mathbb{I}_{pC}(a)$, for some $a\in\Z_{p}^{*}$.

In the following theorem is proved that if a Hadamard matrix with one 
circulant core exists, then its core never is partitioned in $\G_{n}(a)$, this is, if 
$A=(Y_{1},Y_{2},...,Y_{r})$, $Y_{i}\in\G_{n}(a)$, then must be $r=1$.

\begin{theorem}
There is no Hadamard matrices with one circulant core $A$ in $\Z_{2}^{rn}$ with 
$A\in\G_{n}(a)\times\cdots\times\G_{n}(a)\subset\G_{nr}(ar)$, with $nr\equiv3\mod4$.
\end{theorem}
\begin{proof}
Take $A\in\G_{n}(a)\times\cdots\times\G_{n}(a)\subset\G_{nr}(ar)$, $r\geq1$. If
\begin{equation*}
H=
\left( 
\begin{array}{cc}
1 & e \\ 
e^{t} & A_{C}
\end{array}
\right)
\end{equation*}
is a Hadamard matrix, then must be $ar=\frac{nr-1}{2}$. Thus, $a=\frac{nr-1}{2r}$ is not an integer.
Therefore $r=1$.
\end{proof}

We have the corollary for some subgroup $G$ in $Aut(\Z_{2}^{n})$

\begin{corollary}
There is no Hadamard matrices with one circulant core $A$ in $\Z_{2}^{rn}$ with 
$A\in X_{G}\times\cdots\times X_{G}\subset\G_{nr}(ar)$, where $X_{G}$ is the orbit of $X$
under the action of $G$ when $G$ is some of the following groups: $H_{n}\Delta_{n}C_{n}$, 
$\Delta_{n}C_{n}$, $H_{n}C_{n}$, $C_{n}$, $\Delta_{n}$, $H_{n}$, with $nr\equiv3\mod4$.
\end{corollary}

\section{Proof of the equation (\ref{producto})}
For consistency, set $\G_{n}(-1)=\{{}\}$. The formula is trivially true for $n=0,1$ and may be
checked directly for $n=2$. When $a=0$ and for all $n\geq 2$

\begin{equation*}
\G_{n}(0)\G_{n}(b)=\G_{n}(n-b).
\end{equation*}

Suppose $n\geq 3$ and the part I from (\ref{producto}) true for $n-1.$ First, suppose that 
$1\leq a\leq \left[n/2\right]$, and $a\leq b\leq n-a$. Then by Proposition 1 we have that

\begin{eqnarray*}
\G_{n}(a)\G_{n}(b)&=&\left[\left\{+\right\} \times \G_{n-1}(a-1)\cup \left\{-\right\} \times \G_{n-1}(a) \right]\cdot\\
&& \left[ \left\{+\right\} \times \G_{n-1}(b-1)\cup \left\{-\right\} \times \G_{n-1}(b) \right] \\
&=& \left\{+\right\} \times \G_{n-1}(a-1)\G_{n-1}(b-1)\cup \left\{+\right\}\times \G_{n-1}(a)\G_{n-1}(b)\\
&&\cup\left\{-\right\} \times \G_{n-1}(a-1)\G_{n-1}(b)\cup\left\{-\right\}\times \G_{n-1}(a)\G_{n-1}(b-1)\\
&=&\left\{+\right\} \times \bigcup\limits_{i=0}^{a-1}\G_{n-1}(n-a-b+2i+1)\\
&&\cup \left\{+\right\} \times \bigcup\limits_{i=0}^{a}\G_{n-1}(n-a-b+2i-1)\\
&&\cup\left\{-\right\} \times \bigcup\limits_{i=0}^{a-1}\G_{n-1}(n-a-b+2i)
\cup\left\{-\right\}\bigcup\limits_{i=0}^{a}\G_{n-1}(n-a-b+2i)
\end{eqnarray*}

We can simplify the above result showing the first union is contained in the
second union
\begin{eqnarray*}
\bigcup\limits_{i=0}^{a}\G_{n-1}(n-a-b+2i-1)
&=&\bigcup\limits_{i=-1}^{a-1}\G_{n-1}(n-a-b+2i+1) \\
&=&\G_{n-1}(n-a-b-1)\cup\\ 
&&\bigcup\limits_{i=0}^{a-1}\G_{n-1}(n-a-b+2i+1).
\end{eqnarray*}

But in the extreme case $a+b=n$

\begin{eqnarray*}
\bigcup\limits_{i=0}^{a}\G_{n-1}(2i-1)
&=&\bigcup\limits_{i=-1}^{a-1}\G_{n-1}(2i+1) \\
&=&\G_{n-1}(-1)\cup \bigcup\limits_{i=0}^{a-1}\G_{n-1}(2i+1)\\
&=&\bigcup\limits_{i=0}^{a-1}\G_{n-1}(2i+1).
\end{eqnarray*}

And as

\begin{equation*}
\bigcup\limits_{i=0}^{a-1}\G_{n-1}(n-a-b+2i)\subset
\bigcup\limits_{i=0}^{a}\G_{n-1}(n-a-b+2i),
\end{equation*}

then it follows that

\begin{eqnarray*}
\G_{n}(a)\G_{n}(b)&=&\left\{+\right\} \times \bigcup\limits_{i=0}^{a}\G_{n-1}(n-a-b+2i-1)
\cup \left\{-\right\} \times \bigcup\limits_{i=0}^{a}\G_{n-1}(n-a-b+2i)\\
&=&\bigcup\limits_{i=0}^{a}\left[ \left\{+\right\} \times \G_{n-1}(n-a-b+2i-1)\right]
\cup \left[ \left\{-\right\} \times \G_{n-1}(n-a-b+2i)\right] \\
&=&\bigcup\limits_{i=0}^{a}\G_{n}(n-a-b+2i).
\end{eqnarray*}

We will prove the part II from (\ref{producto}) of a similar way. Suppose that $\left[ n/2\right] +1\leq
a\leq n,$ and $n-a\leq b\leq a$. We have
\begin{eqnarray*}
\G_{n}(a)\G_{n}(b)&=&\left[ \left\{+\right\} \times \G_{n-1}(a-1)\cup \left\{-\right\} \times \G_{n-1}(a) \right]\cdot\\
&& \left[\left\{+\right\} \times \G_{n-1}(b-1)\cup \left\{-\right\} \times \G_{n-1}(b) \right] \\
&=& \left\{+\right\} \times \G_{n-1}(a-1)\G_{n-1}(b-1)\cup \left\{+\right\}\times \G_{n-1}(a)\G_{n-1}(b)\\ 
&&\cup \left\{-\right\} \times \G_{n-1}(a-1)\G_{n-1}(b)\cup \left\{-\right\} \times \G_{n-1}(a)\G_{n-1}(b-1)\\	
&=&\left\{+\right\} \times \bigcup\limits_{i=0}^{n-a}\G_{n-1}(a+b-n+2i-1)\\
&&\cup \left\{+\right\} \times \bigcup\limits_{i=0}^{n-a-1}\G_{n-1}(a+b-n+2i+1)\\
&&\cup \left\{-\right\} \times \bigcup\limits_{i=0}^{n-a}\G_{n-1}(a+b-n+2i)
\cup \left\{-\right\} \times \bigcup\limits_{i=0}^{n-a-1}\G_{n-1}(a+b-n+2i).
\end{eqnarray*}
As
\begin{eqnarray*}
\bigcup\limits_{i=0}^{n-a}\G_{n-1}(a+b-n+2i-1)
&=&\bigcup\limits_{i=-1}^{n-a-1}\G_{n-1}(a+b-n+2i+1) \\
&=&\G_{n-1}(a+b-n-1)\cup \\ 
&&\bigcup\limits_{i=0}^{n-a-1}\G_{n-1}(a+b-n+2i+1).
\end{eqnarray*}
And in the extreme case $a+b=n$
\begin{eqnarray*}
\bigcup\limits_{i=0}^{n-a}\G_{n-1}(2i-1)&=&\bigcup\limits_{i=-1}^{n-a-1}\G_{n-1}(2i+1)\\
&=&\G_{n-1}(-1)\cup \bigcup\limits_{i=0}^{n-a-1}\G_{n-1}(2i+1)\\
&=&\bigcup\limits_{i=0}^{n-a-1}\G_{n-1}(2i+1).
\end{eqnarray*}
Also 
\begin{equation*}
\bigcup\limits_{i=0}^{n-a-1}\G_{n-1}(a+b-n+2i)\subset
\bigcup\limits_{i=0}^{n-a}\G_{n-1}(a+b-n+2i).
\end{equation*}
Therefore
\begin{eqnarray*}
\G_{n}(a)\G_{n}(b)&=&\left\{+\right\} \times \bigcup\limits_{i=0}^{n-a}\G_{n-1}(a+b-n+2i-1)\\
&\cup& \left\{-\right\} \times \bigcup\limits_{i=0}^{n-a}\G_{n-1}(a+b-n+2i) \\
&=&\bigcup\limits_{i=0}^{n-a}\left\{+\right\} \times \G_{n-1}(a+b-n+2i-1)\\
&\cup& \bigcup\limits_{i=0}^{n-a}\left\{-\right\} \times \G_{n-1}(a+b-n+2i)\\
&=&\bigcup\limits_{i=0}^{n-a}\G_{n}(a+b-n+2i).
\end{eqnarray*}

\section{Conclution}
In this paper Schur rings induced by permutation automorphic subgroup of $Aut(\Z_{2}^{n})$ were
considered, namely, $\mathfrak{S}(\Z_{2}^{n},S_{n})$, $\mathfrak{S}(\Z_{2}^{n},C_{n})$,
$\mathfrak{S}(\Z_{2}^{n},\Delta_{n})$,$\mathfrak{S}(\Z_{2}^{n},\Delta_{n}C_{n})$,
$\mathfrak{S}(\Z_{2}^{n},H_{n})$, $\mathfrak{S}(\Z_{2}^{n},H_{n}C_{n})$ and 
$\mathfrak{S}(\Z_{2}^{n},H_{n}\Delta_{n}C_{n})$. Also some $S$-sets were defined: 
$\G_{n}(a)$-complete $S$-sets, free and non-free circulant $S$-sets, circulant $S$-sets invariant
by decimation, symmetric, non-symmetric and antisymmetric circulant $S$-sets and the $\delta_{r}$-
invariant $S$-sets. All this issues were relationed with Hadamard matrices. Important results on
Hadamard matrices were obtained:
\begin{enumerate}
\item If a Hadamard matrix exist, then this or some equivalente Hadamard matrix is contained
in a $\G_{4n}(2n)$-complete $S$-set. 
\item Circulant and one core Hadamard matrices of order $4nr$ can't exist if those have 
some particular structure.
\end{enumerate}

\section{Acknowledgment}
We thank an anonymous referees for every careful reading of manuscript and for several suggestions
that improved the presentation.

\Addresses

\end{document}